\documentclass[11pt,a4paper]{amsart}

\RequirePackage[OT1]{fontenc}
\RequirePackage{amsthm,amsmath}
\RequirePackage[numbers]{natbib}
\RequirePackage[colorlinks,citecolor=blue,urlcolor=blue]{hyperref}

\usepackage{graphicx}
\newtheorem{theorem}{Theorem}
\newtheorem{lemma}[theorem]{Lemma}
\newtheorem*{lemma*}{Lemma}
\newtheorem{proposition}[theorem]{Proposition}

\newtheorem{definition}[theorem]{Definition}
\newtheorem{remark}[theorem]{Remark}

\usepackage[utf8]{inputenc}
\usepackage[english]{babel}
\usepackage{latexsym}
\usepackage{amssymb}
\usepackage{amsmath}

\newcommand{\Z}{\mathbb{Z}}
\newcommand{\R}{\mathbb{R}}
\newcommand{\C}{\mathbb{C}}
\newcommand{\E}{\mathbb{E}}

\newcommand{\Prob}{\mathbb{P}}

\begin{document}

\title[Global spectral fluctuations in the GUE]{Global spectral fluctuations in the Gaussian Unitary Ensemble}
\author{Christian Webb}
\address{Department of mathematics and systems analysis, Aalto University, PO Box 11000, 00076 Aalto, Finland}
\email{christian.webb@aalto.fi}
\date{\today}

\begin{abstract}
We consider global fluctuations of the spectrum of the GUE. Using results on the linear statistics of such matrices as well as variance bounds on the eigenvalues, we show that under a suitable scaling, global fluctuations of the spectrum can be asymptotically described in terms of a logarithmically correlated Gaussian field. We also discuss briefly connections between different objects in RMT giving rise to log-correlated fields: the logarithm of the absolute value of the characteristic polynomial, the eigenvalue counting function, and the field of fluctuations of the eigenvalues around their expected locations.  
\end{abstract}

\maketitle

\section{Introduction}

The goal of this note is to describe the fluctuations (around the semicircle law) of the eigenvalues of a large random matrix drawn from the Gaussian Unitary Ensemble. Such questions have of course been  studied previously. For example, Gustavsson has proven in \cite{gustavsson} that if one considers a single eigenvalue in the bulk of the spectrum and its fluctuation around its "classical location" - i.e. that suggested by the semicircle law - and normalizes by its variance, this quantity is asymptotically a standard normal random variable. Gustavsson also provides similar results near the edge and also for the joint distribution of a fixed number of eigenvalues.

We aim to consider a similar question, but with the difference that we consider the global scale of the spectrum - namely we consider simultaneously a positive fraction of all of the eigenvalues. To obtain a non-trivial limiting object, the normalization will be different from that in \cite{gustavsson}. Also the limiting object will be a rather rough object - a logarithmically correlated Gaussian field that needs to be understood as a generalized function. 

Global fluctuations are usually described in terms of linear statistics \cite{ds,johansson,rv} (or essentially equivalently: in terms of the logarithm of the absolute value of the characteristic polynomial \cite{fks,hko,rv} or the eigenvalue counting function \cite{bg}). While the essence of our main result follows from a rather simple calculation relating the global fluctuations of the eigenvalues around their expected locations to linear statistics (in particular, making use of results from \cite{dallaporta,johansson}) the author has not run into a rigorous statement in the literature concerning this connection or the description of the limiting object in terms of a log-correlated Gaussian field. That being said, results of this flavor are known either in other models or on a non-rigorous level: such a result is known to hold for the Plancherel measure on Young diagrams (see \cite{io}), and in the context of RMT (non-rigorous) results of this type were indicated in \cite{rgmrf}. One reason to be interested in such a limit theorem is that recently, there have been advances in understanding how the extrema of approximations of log-correlated fields behave (see e.g. \cite{drz}), and one might suspect that if such behavior is universal, this could suggest how the maximal fluctuation of the eigenvalues (in the bulk) behavs.

The precise thing we will prove is that the global "fluctuation field" converges weakly in a certain Sobolev space to log-correlated Gaussian field (an approach taken also in \cite{fks,hko} when studying the logarithm of the characteristic polynomial). If one were to consider only convergence in the sense of finite dimensional distributions (instead of weak convergence), the proof would be quite a bit simpler - indeed weak convergence requires bounds on the growth rate of a certain family of linear statistics and we'll need to be fairly specific about the space our object lives on. One of the main motivation for not being satisfied with convergence in the sense of finite dimensional distributions is a connection between the fluctuation field and the eigenvalue counting function. Our analysis for weak convergence suggests a (fairly weak) way to analyze the asymptotic difference between these two objects. With stronger estimates, it seems possible that one could relate for example the extrema of these two fields to one another. The drawback of our approach is that it unfortunately makes this note quite heavy even though the argument at the heart of our results is a very simple calculation.

The outline of the paper is the following: we'll begin with introducing the relevant notation to define our main object of interest. Before defining our limiting object and stating our main result, we'll have to introduce the relevant Sobolev spaces where our objects live, and discuss some properties of Chebyshev polynomials (which are relevant for the description of our Sobolev spaces). After stating our main result, we'll recall some facts from the literature: e.g. Johansson's result on the linear statistics of the GUE (\cite{johansson}) as well as discuss some bounds on the fluctuations of the eigenvalues due to Dallaporta \cite{dallaporta} which will play an important role in our proof.  After this, we'll give a heuristic proof of our main result and discuss non-rigorously some of the connections between the different objects in the GUE giving rise to log-correlated Gaussian fields. Finally we'll prove our main result.

We mention here that throughout this note, we use the convention that the value of (irrelevant) constants may change from line to line, and we won't comment on this further. 

\vspace{0.3cm}

{\bf Acknowledgements:} The author wishes to thank Y.V. Fyodorov for helpful comments and pointing out the reference \cite{rgmrf}.

\section{Notation, main result, and some results from the literature}

In this section, we'll first fix our notation related to the GUE and define the object we are interested in. After this, we'll describe the Sobolev spaces where our object lives in. Next we'll describe our main result. Finally we'll recall some known results that will play important roles in our argument.

\subsection{The GUE and the "fluctuation field"}

Let us first fix our notation concerning the Gaussian Unitary Ensemble and its eigenvalues.

\begin{definition}\label{def:gue}
For $N\in \Z_+$, the {\rm{GUE($N$)}} is the following probability measure on the space of $N\times N$ complex Hermitian matrices:

\begin{equation}
\Prob_N(dH)=\frac{1}{Z_N}e^{-\mathrm{Tr} H^2}dH,
\end{equation}

\noindent where $Z_N$ is an explicit constant we don't care about and $dH$ is Lebesgue measure on the space of $N\times N$ Hermitian complex matrices. 

We write $\mathcal{H}=\frac{1}{\sqrt{2N}}H$ and $(\lambda_1,...,\lambda_N)$ for the ordered eigenvalues of $\mathcal{H}$: $\lambda_1\leq \lambda_2\leq \cdots \leq \lambda_N$.
\end{definition}

We'll also introduce some notation related to the semicircle law.

\begin{definition}\label{def:semic}
For $x\in[-1,1]$, let 

\begin{equation}
\begin{array}{ccc}
\sigma(x)=\frac{2}{\pi}\sqrt{1-x^2} & \mathrm{and} & G(x)=\int_{-1}^x\sigma(y)dy.
\end{array}
\end{equation}

Moreover, for $j\in\lbrace 0,1,...,N\rbrace$, let 

\begin{equation}
\gamma_j=G^{-1}\left(\frac{j}{N}\right).
\end{equation}
\end{definition}

\begin{remark}
It follows from \cite{gustavsson} that the points $\gamma_j$  - sometimes called the "classical locations" of the eigenvalues $\lambda_j$ - are close to the expectation of the eigenvalues $\gamma_j$. Though as suggested by the following definition, they are not quite close enough  for our needs.
\end{remark}

As already discussed, we are interested in studying (more or less) all of the eigenvalues $(\lambda_j)_{j=1}^N$ simultaneously. Instead of encoding them into the discrete index $j$, we'll find it more convenient to introduce a continuous index - otherwise it wouldn't be so clear what space our object lives in or what kind of limit it converges to. Also we want to study fluctuations around the semi circle law and scale things in such a way that a non-trivial limit exists. The following definition of the "eigenvalue fluctuation field" will turn out to be the correct one for our purposes.

\begin{definition}\label{def:field}
For $j=1,...,N$ and $x\in(\gamma_{j-1},\gamma_j]$, let 

\begin{equation}
X_N(x)=N\sigma(\gamma_j)\left(\lambda_j-N\int_{\gamma_{j-1}}^{\gamma_j}y\sigma(y)dy\right).
\end{equation}
\end{definition} 

\begin{remark}
Note that the factor $N\sigma(\gamma_j)$ is roughly the inverse of the average eigenvalue gap $\E(\lambda_{j+1}-\lambda_j)$ in the sense that when one normalizes the eigenvalue gap by this quantity, one has convergence (to the Gaudin-Mehta law) as $N\to\infty$ - see \cite{tao}. We also point out that it follows directly from our definitions that

\begin{equation}
\gamma_{j-1}\leq N\int_{\gamma_{j-1}}^{\gamma_j}y\sigma(y)dy\leq \gamma_j,
\end{equation}

\noindent but shifting $\lambda_j$ just by either $\gamma_j$ or $\gamma_{j-1}$ would not produce a centered field.
\end{remark}

\subsection{Sobolev spaces}

It will turn out that it is natural to expand our field $X_N$ in Chebyshev polynomials of the second kind so we'll need to discuss the space where such an expansion will be a well defined object. We'll begin with recalling the definition of Chebyshev polynomials of the second kind and some basic facts related to them. Then we'll define our space, which is a Sobolev space of generalized functions. We'll also point out an identity the Chebyshev polynomials of the second kind satisfy which will be useful in describing our limiting object appearing in our main result. Finally we'll recall some basic facts about Chebyshev polynomials of the first kind that we'll need in our proof.

\begin{definition}\label{def:cheby}
For $k\in\lbrace 0,1,...\rbrace$, the Chebyshev polynomial of the second kind of degree $k$ is defined (on $(-1,1)$) by 

\begin{equation}
U_k(\cos \theta)=\frac{\sin (k+1)\theta}{\sin \theta}.
\end{equation}
\end{definition}

\begin{remark}\label{rem:ortho}
Using basic trigonometric identities it is simple to check that $U_0(x)=1$, $U_1(x)=2x$, and 

\begin{equation}
U_{k+1}(x)=2xU_k(x)-U_{k-1}(x)
\end{equation}

\noindent so $U_k$ really is a polynomial.

\vspace{0.3cm}

One can also check from the orthogonality of trigonometric functions that 

\begin{equation}
\int_{-1}^1 U_k(x) U_l(x)\sigma(x)dx=\delta_{k,l}.
\end{equation}
\end{remark}

Let us now define the Fourier-Chebyshev coefficients (of the second kind).

\begin{definition}\label{def:chebcoef}
Let 

\begin{equation}
\nu(dx)=\frac{2}{\pi}\frac{1}{\sqrt{1-x^2}}dx,
\end{equation}

\noindent $f\in L^2((-1,1),\nu)$ and $k\in\lbrace 0,1,...\rbrace$. We then define 

\begin{equation}
s_k(f)=\frac{2}{\pi}\int_{-1}^1 f(x)U_k(x)dx.
\end{equation}
\end{definition}

\begin{remark}\label{rem:sineseries}
Note that if we use the coordinate $x=\cos \theta$ and write $\widetilde{f}(\theta)=f(\cos \theta)$, we have

\begin{equation}
s_k(f)=\frac{2}{\pi}\int_0^\pi \widetilde{f}(\theta)\sin (k+1)\theta d\theta
\end{equation}

\noindent so we see that $(s_k(f))_{k=0}^\infty$ are simply the coefficients of $\widetilde{f}$ in its Fourier sine series. Thus for example if $f$ is differentiable and $f(-1)=f(1)=0$, then

\begin{equation}
f(\cos \theta)=\sum_{k=0}^\infty s_k(f)\sin(k+1)\theta=\sum_{k=0}^\infty s_k(f)U_k(\cos \theta)\sin \theta
\end{equation}

\noindent or if we write $x=\cos \theta$, 

\begin{equation}
f(x)=\sum_{k=0}^\infty s_k(f)U_k(x)\sqrt{1-x^2}.
\end{equation}

Moreover, the role of $\nu$ in Definition \ref{def:chebcoef} was so that $L^2((-1,1),\nu)$ is the space where the theory of Chebyshev series (of the second kind) is analogous to standard $L^2$ Fourier theory. In fact, many of the definitions we consider would perhaps be simpler to state and more natural in the "trigonometric coordinates", but the technical calculations we'll carry out in the proof are simpler with the current definitions.

\end{remark}

One of the most common Sobolev spaces consists of $L^2$ functions on the unit circle whose Fourier coefficients satisfy certain decay properties. Keeping in mind Remark \ref{rem:sineseries}, we'll want to restrict to the "sine part" of this space so we make the following definition (though the relevant interpretation is in Remark \ref{rem:intep}).

\begin{definition}\label{def:sobo}
For $\alpha\in \R$, let 

\begin{equation}
\mathcal{S}_\alpha=\left\lbrace s=(s_k)_{k=0}^\infty\in \R^\infty: \sum_{k=0}^\infty|s_k|^2(1+k^2)^\alpha<\infty\right\rbrace
\end{equation}

\noindent and equip it with the inner product

\begin{equation}
\langle s,s'\rangle_\alpha=\sum_{k=0}^\infty s_k s_k'(1+k^2)^\alpha.
\end{equation}

We write $||\cdot ||_\alpha$ for the corresponding norm.

\end{definition}

\begin{remark}\label{rem:intep}
As in the Fourier case, $\mathcal{S}_\alpha$ is a Hilbert space and can be identified with a subspace of $L^2((-1,1),\nu)$ if $\alpha>0$ and in this case, $\mathcal{S}_{-\alpha}$ can be understood as the dual space (consisting of generalized functions) of $\mathcal{S}_\alpha$. The interpretation is that for $\alpha>0$ we understand the elements of $\mathcal{S}_\alpha$ to be functions of the form

\begin{equation}
f(x)=\sum_{k=0}^\infty s_k(f)U_k(x)\sqrt{1-x^2},
\end{equation}

\noindent where 

\begin{equation}
\sum_{k=0}^\infty |s_k(f)|^2(1+k^2)^\alpha<\infty.
\end{equation}

The elements of $\mathcal{S}_{-\alpha}$ are then generalized functions $\psi$ which we understand as the formal series 

\begin{equation}
\psi(x)=\sum_{k=0}^\infty s_k(\psi)U_k(x)\sqrt{1-x^2}
\end{equation}

\noindent satisfying 

\begin{equation}
\sum_{k=0}^\infty |s_k(\psi)|^2(1+k^2)^{-\alpha}.
\end{equation}

The action of $\psi\in \mathcal{S}_{-\alpha}$ on $f\in \mathcal{S}_{\alpha}$ is 

\begin{equation}
\psi(f)=\sum_{k=0}^\infty s_k(\psi)s_k(f)
\end{equation}

\noindent which can be written formally as 

\begin{equation}
\psi(f)=\int_{-1}^1 \psi(x)f(x)\frac{2}{\pi}\frac{1}{\sqrt{1-x^2}}dx.
\end{equation}

\end{remark}

We also point out the following fact which will be useful when describing the covariance kernel of our limiting object.

\begin{lemma}\label{le:logdiag}
Let $x,y\in (-1,1)$. Then 

\begin{align}
\notag C(x,y):&=\sum_{k=0}^\infty \frac{1}{k+1}U_k(x)U_k(y)\sqrt{1-x^2}\sqrt{1-y^2}\\
&=-\log \frac{|x-y|}{1-xy+\sqrt{1-x^2}\sqrt{1-y^2}}.
\end{align}
\end{lemma}

\begin{proof}
Write $x=\cos \theta$ and $y=\cos \phi$ so we have 

\begin{align}
\notag C(x,y)&=\sum_{k=1}^\infty \frac{1}{k}\sin k\theta\sin k\phi\\
\notag &=\sum_{k=1}^\infty \frac{1}{2k}(\cos k(\theta-\phi)-\cos k(\theta+\phi))\\
&=-\frac{1}{2}\log \left(4\sin^2\frac{\theta-\phi}{2}\right)+\frac{1}{2}\log \left(4\sin^2\frac{\theta+\phi}{2}\right)\\
\notag &=-\log \frac{\left|\sin\frac{\theta-\phi}{2}\right|}{\left|\sin\frac{\theta+\phi}{2}\right|}.
\end{align}

We then note that as

\begin{equation}
\sin^2\frac{x}{2}=\frac{1}{2}(1-\cos x),
\end{equation}

\noindent we have 

\begin{align}
\notag \frac{\left|\sin\frac{\theta-\phi}{2}\right|}{\left|\sin\frac{\theta+\phi}{2}\right|}&=\frac{2\left|\sin\frac{\theta-\phi}{2}\sin\frac{\theta+\phi}{2}\right|}{1-\cos(\theta+\phi)}\\
&=\frac{|\cos \theta-\cos\phi|}{1-\cos(\theta+\phi)}\\
\notag &=\frac{|x-y|}{1-xy+\sqrt{1-x^2}\sqrt{1-y^2}}
\end{align}

\noindent which completes the proof.

\end{proof}

We'll also make use of the following property of the Chebyshev polynomials of the first kind (which is actually a special case of a general result due to Markov - see e.g. \cite{shadrin} for proofs and further information).

\begin{proposition}\label{prop:markov}
Let $T_n$ be a Chebyshev polynomial of the first kind (i.e. it satisfies $T_n(\cos\theta)=\cos n\theta$). Then 

\begin{equation}
\sup_{x\in[-1,1]}|T_n^{(k)}(x)|\leq T_n^{(k)}(1)=\prod_{j=0}^{k-1}\frac{n^2-j^2}{(2j+1)}.
\end{equation}

\end{proposition}

We also point out the following simple results.

\begin{lemma}\label{le:chebyder}
For any $n\in\lbrace 0,1,...\rbrace$ and $k\in \Z_+$,

\begin{equation}
x\mapsto T_n^{(k)}(x)
\end{equation}

\noindent is increasing for $x>1$.
\end{lemma}

\begin{proof}
We write 

\begin{equation}
T_n(x)=\sum_{j=0}^n\frac{T_n^{(j)}(1)}{j!}(x-1)^j
\end{equation}

\noindent so that 

\begin{equation}
T_n^{(k)}(x)=\sum_{j=k}^n\frac{1}{(j-k)!}T_n^{(j)}(1)(x-1)^{j-k}.
\end{equation}

From Proposition \ref{prop:markov} we see in particular that $T_n^{(j)}(1)$ are non-negative. So we see that $T_n^{(k)}(x)$ is non-negative for $x\geq 1$ for all $k$ which implies in particular that $T_n^{(k+1)}(x)\geq 0$ for $x\geq 1$ so $T_n^{(k)}(x)=T_n^{(k)}(x)$ is increasing.

\end{proof}

\begin{lemma}\label{le:tder}
For $x>1$,

\begin{equation}
T_k''(x)\leq k^4T_k(x).
\end{equation}
\end{lemma}

\begin{proof}
Let us expand 

\begin{equation}
T_k(x)=\sum_{j=0}^k \frac{T_k^{(j)}(1)}{j!}(x-1)^j.
\end{equation}

We thus have 

\begin{equation}
T_k''(x)=\sum_{j=2}^k \frac{T_k^{(j)}(1)}{(j-2)!}(x-1)^{j-2}=\sum_{j=0}^{k-2}\frac{T_k^{(j+2)}(1)}{j!}(x-1)^j.
\end{equation}

From Proposition \ref{prop:markov}, 

\begin{equation}
T_k^{(j+2)}(1)=\prod_{p=0}^{j+1}\frac{k^2-p^2}{2p+1}=T_k^{(j)}(1)\frac{k^2-j^2}{2j+1}\frac{k^2-(j+1)^2}{2j+3}\leq k^4 T_k^{(j)}(1).
\end{equation}

Thus for $x>1$ (as $T_k^{(l)}(1)\geq 0$ for all $l$)

\begin{equation}
T_k''(x)\leq k^4 \sum_{j=0}^{k-2}\frac{T_k^{(j)}(1)}{j!}(x-1)^j\leq k^4 \sum_{j=0}^{k}\frac{T_k^{(j)}(1)}{j!}(x-1)^j=k^4 T_k(x).
\end{equation}

\end{proof}

\begin{lemma}\label{le:tbound}
For $x>1$

\begin{equation}
T_k(x)\leq e^{2k \sqrt{x-1}}.
\end{equation}
\end{lemma}

\begin{proof}
Recall that $T_k$ satisfies

\begin{equation}
T_k(\cos \theta)=\cos k\theta
\end{equation}

\noindent for $\theta\in[0,\pi]$. Both sides of this equation are entire functions of $\theta$ so this equality holds for all $\theta\in \C$. In particular, for $\theta$ on the imaginary axis, if we write $\theta=i\tau$, we have 

\begin{equation}
T_k(\cosh \tau)=\cosh k\tau
\end{equation}

\noindent for all $\tau\in \R$. For $\tau>0$, write $\tau=\mathrm{arcosh}\, x$, where $x>1$, and we find

\begin{equation}
T_k(x)=\cosh k\:\mathrm{arcosh}\, x\leq e^{k\: \mathrm{arcosh}\, x}.
\end{equation}

Consider then the function $h(x)=\mathrm{arcosh}\, x-2\sqrt{1-x}$. One has $h(1)=0$ and 

\begin{equation}
h'(x)=\frac{1-\sqrt{x+1}}{\sqrt{x^2-1}}<0
\end{equation}

\noindent so $h(x)\leq 0$ for $x\geq 1$ and we have our claim.

\end{proof}

\subsection{Main result}

We now define our limiting object and state our main result as well as discuss the relevant weak convergence theory on our Sobolev spaces $\mathcal{S}_{-\alpha}$.

\begin{definition}
Let $(Y_k)_{k=0}^\infty$ be i.i.d. standard Gaussian random variables, and let 

\begin{equation}
X(x)=\sum_{k=0}^\infty\frac{1}{\sqrt{k+1}}Y_k U_k(x)\sqrt{1-x^2},
\end{equation}

\noindent or alternatively identify $X$ with the sequence

\begin{equation}
(s_k(X))_{k=0}^\infty=\left(\frac{1}{\sqrt{k+1}}Y_k\right)_{k=0}^\infty
\end{equation}

\noindent which almost surely is an element of $\mathcal{S}_{-\alpha}$ for any $\alpha>0$ (since $\E||X||_{-\alpha}^2<\infty$).
\end{definition}

\begin{remark}
Again it might be clearer to define $X$ in the trigonometric coordinates, where one would just have

\begin{equation}
X(\cos \theta)=\sum_{k=1}^\infty \frac{1}{\sqrt{k}}Y_{k-1}\sin k\theta.
\end{equation}
\end{remark}

\begin{remark}\label{rem:logcor}
Note that by Lemma \ref{le:logdiag}, we have (formally - though making this precise would not require much work)

\begin{equation}
\E(X(x)X(y))=-\log \frac{|x-y|}{1-xy+\sqrt{1-x^2}\sqrt{1-y^2}}.
\end{equation} 

Thus in the bulk, when $x\to y$, the covariance has a logarithmic singularity:

\begin{equation}
\E(X(x)X(y))= -\log |x-y|+\mathcal{O}(1).
\end{equation}

\end{remark}

We are now in a position to state our main result.

\begin{theorem}\label{th:main}
For $\alpha>3$, $\pi X_N$ converges to $X$ in distribution in the space $\mathcal{S}_{-\alpha}$.
\end{theorem}

\begin{remark}
Our bound on $\alpha$ here is likely to be far from optimal. Indeed, the limiting object is in $\mathcal{S}_{-\alpha}$ for any $\alpha>0$. As $\alpha$ is simply a measure of roughness of the field $X_N$ we will not spend energy on optimizing this.
\end{remark}

\begin{remark}
Note that this result makes precise the statement from \cite{rgmrf} that the global spectral fluctuations should be asymptotically "1/f noise" - in particular when one considers the representation of the field in trigonometric coordinates.
\end{remark}

\begin{remark}
It's reasonable to expect that corresponding results might hold for other ensembles as well: e.g. circular ensembles, $\beta$-ensembles, and Wigner matrices. It might also be an interesting question to consider what is the correct object to study in the case where the eigenvalues live in a two-dimensional space such as in the case of the Ginibre ensemble. Another interesting question might be to try to prove such a result directly  for example from the determinantal structure of the eigenvalues $(\lambda_j)$. In \cite{dallaporta,gustavsson} only variances of the eigenvalues are considered, but perhaps one is able to use similar methods to consider other moments as well.
\end{remark}

Let us discuss briefly what convergence in distribution in the space $\mathcal{S}_{-\alpha}$ means. Recall that in general, weak convergence of random variables, say $\xi_n$ to $\xi$, on a Polish space (separable complete metric space) say $S$ (such as $\mathcal{S}_{-\alpha}$) means that $\E(f(\xi_n))\to \E(f(\xi))$ for each bounded continuous function $f:S\to \R$. 

Using this definition is often not so practical for proving convergence. Instead, one makes use of Prohorov's theorem (see e.g. \cite{billingsley}). From Prohorov's theorem, it follows that if the random variables $(\xi_n)$ converge on suitable cylinder sets and the sequence $(\xi_n)$ is tight, then $(\xi_n)$ converges in distribution. Thus to prove our theorem, we'll need to prove convergence on some suitable cylinder sets, and tightness.

\subsection{Asymptotics of linear statistics}

Our main tool will be a result of Johansson (\cite[Theorem 2.4]{johansson}) describing the asymptotic properties of linear statistics of the GUE (in fact more general $\beta$-ensembles). We'll only need a simplified version of the theorem.

\begin{theorem}[Johansson]\label{th:johansson}
Let $f\in C_c^\infty(\R)$. Then as $N\to\infty$,

\begin{equation}
\sum_{j=1}^N f(\lambda_j)-N\int_{-1}^1 f(x)\sigma(x)dx\stackrel{d}{\to} V_f,
\end{equation}

\noindent where $V_f$ is a centered Gaussian random variable with variance

\begin{equation}
\rho_f=\frac{1}{4}\sum_{k=1}^\infty k a_k^2,
\end{equation}

\noindent where 

\begin{equation}
a_k=\frac{2}{\pi}\int_0^\pi f(\cos \theta)\cos k\theta d\theta.
\end{equation}
\end{theorem}

Let us point out that the variance can also be written in terms of the Fourier-Chebyshev coefficients (of the second kind) of $f'(x)\sqrt{1-x^2}$.

\begin{lemma}\label{le:var}
Let $f$ be as in Theorem \ref{th:johansson}. Then 

\begin{equation}
\rho_f=\frac{1}{4}\sum_{k=0}^\infty\frac{1}{k+1}s_k(f'(x)\sqrt{1-x^2})^2
\end{equation}
\end{lemma}

\begin{proof}
In the definition of $a_k$, let us integrate by parts and make a change of variables ($x=\cos \theta$):

\begin{align}
\notag a_k&=\frac{1}{k}\frac{2}{\pi}\int_0^\pi f(\cos \theta)\frac{d}{d\theta}\sin k\theta d\theta\\
&=\frac{1}{k}\frac{2}{\pi}\int_0^\pi f'(\cos \theta)\sin\theta U_{k-1}(\cos \theta)\sin \theta d\theta\\
\notag &=\frac{1}{k}\frac{2}{\pi}\int_{-1}^1f'(x)U_{k-1}(x)\sqrt{1-x^2}dx\\
\notag &=\frac{1}{k}s_{k-1}(f'(x)\sqrt{1-x^2}).
\end{align}

This yields the claim.
\end{proof}

\subsection{Moment bounds}

Another important ingredient of our proof is a bound on moments of the form $\E|\lambda_j-\gamma_j|^k$ for $k=2$ and $k=4$. For $k=2$ these were studied in \cite{dallaporta}, and (as is pointed out in \cite{dallaporta}) the $k=4$ case is proven with essentially identical arguments as the $k=2$ case and we won't give proofs here.

\begin{proposition}\label{prop:dallaporta}
Let $k=2$ or $k=4$.

\begin{itemize}

\item[1.] Let $\delta\in(0,1/2]$. There exists a $C=C(\delta)$ such that for $j\in(\delta N,(1-\delta)N)$

\begin{equation}
\E|\lambda_j-\gamma_j|^k\leq C\frac{(\log N)^{k/2}}{N^k}.
\end{equation}

\item[2.] Let $\delta\in(0,1/2]$ and $K>20\sqrt{2}$. Then there exists a $C=C(\delta,K)$ such that for $j\in[K\log N,\delta N]$ and $j\in[(1-\delta)N,N-K\log N]$

\begin{equation}
\E|\lambda_j-\gamma_j|^k\leq C\frac{(\log \min(j,N-j))^{k/2}}{N^{\frac{2k}{3}}\min(j^{\frac{k}{3}},(N-j)^{\frac{k}{3}})}.
\end{equation}

\item[3.] There exists a $C=C(K)$ such that for $j\in[1,K\log N)$ and $j\in[N-K\log N,N]$,

\begin{equation}
\E|\lambda_j-\gamma_j|^k\leq C\frac{1}{N^{k/2}}.
\end{equation}

\item[4.] Combining these, there exists a $C>0$ such that 

\begin{equation}
\sum_{j=1}^N \E|\lambda_j-\gamma_j|^k\leq C\frac{\log N}{N^{k/2}}.
\end{equation}

\end{itemize}

\end{proposition}

\begin{remark}
For $k=2$, part 1 of Proposition \ref{prop:dallaporta} is \cite[Theorem 5]{dallaporta}, part 2 of Proposition \ref{prop:dallaporta} is \cite[Theorem 9]{dallaporta}. Part 3 of Proposition \ref{prop:dallaporta} does not appear as a separate result, but its proof (relying on localization results such as \cite[Theorem 2.2]{eyy} or \cite[Corollary 15]{tv}) is part of the proof of \cite[Corollary 4]{dallaporta}.
\end{remark}

\subsection{Estimates for eigenvalues lying outside of the support of the equilibrium measure} We'll need some control on events of the eigenvalues $\lambda_j$ lying outside of $[-1,1]$. In particular, we'll make use of the following ones.

\begin{proposition}\label{prop:density}
Let $\rho_N$ be the density of the mean (normalized) eigenvalue counting measure, i.e. the function satisfying 

\begin{equation}
\sum_{j=1}^N \E f(\lambda_j)=\int_\R f(x)N\rho_N(x)dx
\end{equation}

\noindent for nice enough functions $f$ (see e.g. \cite{agz,ps} for more information).

\begin{itemize}
\item[1.] For $x>1$ 

\begin{align}
\notag N\rho_N(1+x)&=\left(\frac{\Phi'(x)}{4\Phi(x)}-\frac{\gamma'(x)}{\gamma(x)}\right) 2 \mathrm{Ai}(N^{2/3}\Phi(x))\mathrm{Ai}'(N^{2/3}\Phi(x))\\
&\quad + N^{2/3}\Phi'(x)\left(\mathrm{Ai}'(N^{2/3}\Phi(x))^2-N^{2/3}\Phi(x)\mathrm{Ai}(N^{2/3}\Phi(x))^2\right)\\
&\notag \quad +\mathcal{O}\left(\frac{1}{N\sqrt{x-1}}\right),
\end{align}

\noindent where $\mathrm{Ai}$ is the Airy function, 

\begin{equation}
\gamma(x)=\left(\frac{x-1}{x+1}\right)^{1/4}
\end{equation}

\noindent and (again for $x>1$)

\begin{equation}
\Phi(x)=\left(3\int_1^x\sqrt{y^2-1}dx\right)^{2/3}.
\end{equation}

\item[2.] The bound 

\begin{equation}
\rho_N\left(1+\frac{s}{N^{2/3}}\right)\leq \frac{1}{B s N^{1/3}}e^{-bs^{3/2}}
\end{equation}

\noindent is valid for some absolute constants $B$ and $b$ and $s\to \infty$, $n\to\infty$.

\item[3.] The bound 

\begin{equation}
\begin{array}{ccc}
\rho_N(x)\leq C e^{-cN x^2}, & \mathrm{for} & x>(1+\delta)
\end{array}
\end{equation}

\noindent is valid for any $\delta>0$ and some $C<\infty$ and $c>0$ possibly depending on $\delta$.
\end{itemize}
\end{proposition}

\begin{remark}
Here part 1 follows from \cite[equation (4.4)]{em} (see also \cite[Proof of Lemma 2.2]{gustavsson}). Part 2 is \cite[Theorem 5.2.3 (ii)]{ps} and part 3 is \cite[Theorem 5.2.3 (iii)]{ps}. We also note that as $\rho_N$ is a even function, such bounds also hold near the point $-1$.
\end{remark}

\section{Some heuristic comments about log-correlated Gaussian fields and random matrix theory}

In this section, we discuss briefly different objects in RMT giving rise to log-correlated Gaussian fields. We also give a brief non-rigorous proof of our main result and point out a few connections between the different objects we discuss here.

\subsection{The logarithm of the absolute value of the characteristic polynomial}

For a large class of random matrix theory models (see for example \cite{fks,hko,rv}), it is known that if $(z_1,...,z_N)$ are the eigenvalues (normalized in a suitable way) and $z$ is a point in the support of the equilibrium measure, 

\begin{equation}
z\mapsto \mathcal{X}_N(x)=\sum_{j=1}^{N}\left(\log |z-z_j|-\E\log |z-z_j|\right)
\end{equation}

\noindent converges (in some sense) to a log-correlated Gaussian field. Such results are closely related to the asymptotically Gaussian fluctuations of the linear statistics. For example, for the CUE, if  $U_N$ is a $N\times N$ Haar distributed unitary matrix,$(e^{i\theta_j})_{j=1}^{N}$ are the eigenvalues of $U_N$, and $e^{i\theta}$ is a fixed point on the unit circle, one has 

\begin{equation}
\sum_{j=1}^{N}\log |e^{i\theta}-e^{i\theta_j}|=\frac{1}{2}\sum_{j=1}^{\infty}\frac{1}{j}\left(\mathrm{Tr}(U_N^{j})e^{-ij\theta}+ \mathrm{Tr}(U_N^{-j})e^{ij\theta}\right).
\end{equation}

It was proven in \cite{ds}, that for any fixed $K\in \Z_+$, $(j^{-1/2}\mathrm{Tr}(U_N^{j}))_{j=1}^{K}$ converge to i.i.d. standard complex Gaussian random variables. This suggests that in this case, $\mathcal{X}_N$ converges (in some sense) to a random Fourier series with coefficients that are independent Gaussians with variance proportional to $j^{-1}$. This is a log-correlated Gaussian field. This was indeed proven in \cite{hko}.

\vspace{0.3cm}

Similar diagonalizations of $\log|z-w|$ can be applied for other ensembles. For the GUE, if $x,y\in[-1,1]$, one has 

\begin{equation}
-\log(2|x-y|)=\sum_{n=1}^{\infty}\frac{2}{n}T_n(x)T_n(y) 
\end{equation}

\noindent and in \cite{fks}, Johansson's result (Theorem \ref{th:johansson}) was used to prove that again in this case $\mathcal{X}_N$ converges to a log-correlated field. 

\vspace{0.3cm}

The linear statistics of the Ginibre ensemble were studied in \cite{rv} and here it was also proven that $\mathcal{X}_N$ can be related to the two-dimensional Gaussian Free Field.

\subsection{The imaginary part of the characteristic polynomial and the eigenvalue counting function}

If $x,y\in \R$, one has (with a suitable choice of the branch of the logarithm)

\begin{equation}
\mathrm{Im}\log (x-y)=\begin{cases}
\pi , & x<y\\
0, & x>y\\
-\infty, & x=y.
\end{cases}
\end{equation}

Thus (suitably interpreted), one has for example for the GUE

\begin{equation}
\mathrm{Im}\sum_{j=1}^{N}\log  (\lambda_j-x)=\pi \sum_{j=1}^{N}\mathbf{1}(\lambda_j<x)
\end{equation}

\noindent if $x\neq \lambda_j$ for all $j$. So the imaginary part of the logarithm of the characteristic polynomial is related to the eigenvalue counting function (or the height function as its called in \cite{bg}). Consider now the centered field 

\begin{equation}
\mathbb{X}_N(x)=\sum_{j=1}^{N}(\mathbf{1}(\lambda_j<x)-\mathbb{P}(\lambda_j<x)).
\end{equation}

As pointed out in \cite{bg}, this is asymptotically a log-correlated Gaussian field (when studied on the support of the equilibrium measure). To see the connection to $X_N$ consider now a test function $f\in C_c^{\infty}(\R)$ and the action of $\mathbb{X}_N$ on $f'$. Integrating by parts we have 

\begin{equation}
\int_\R\mathbb{X}_N(x)f'(x)dx=\sum_{j=1}^{N}(f(\lambda_j)-\E f(\lambda_j)),
\end{equation}

\noindent which is asymptotically a Gaussian random variable. Specializing to the case where $f(x)=(k+1)^{-1}T_{k+1}$  and ignoring the contribution of the eigenvalues outside of the interval $[-1,1]$ we might expect to have for $x\in(-1,1)$

\begin{equation}
\mathbb{X}_N(x)\sim \sum_{k=0}^{\infty}\frac{1}{k+1}\sum_{j=1}^{N}(T_{k+1}(\lambda_j)-\E T_{k+1}(\lambda_j))U_{k-1}(x)\sqrt{1-x^{2}}.
\end{equation}

\subsection{The eigenvalue fluctuation field} 

We'll now give a heuristic proof for the convergence of $X_N$ to a log-correlated Gaussian field. Consider the setting of Theorem \ref{th:johansson}. Another way to write the relevant random variable is (simply by Taylor expanding $f$ around $\gamma_j$ in the two terms)

\begin{align}
\notag \sum_{j=1}^N & f(\lambda_j)-N\int_{\gamma_{j-1}}^{\gamma_j}f(x)\sigma(x)dx\\
&=\sum_{j=1}^N (f(\gamma_j)+f'(\gamma_j)(\lambda_j-\gamma_j)+\mathcal{O}(|\lambda_j-\gamma_j|^2))\\
\notag &\qquad -\sum_{j=1}^N \left(f(\gamma_j)+f'(\gamma_j)N\int_{\gamma_{j-1}}^{\gamma_j}(x-\gamma_j)\sigma(x)dx+\mathit{o}(|\gamma_j-\gamma_{j-1}|)\right)\\
\notag &=\frac{1}{N}\sum_{j=1}^N f'(\gamma_j)\frac{1}{\sigma(\gamma_j)}X_N(\gamma_j)+\mathrm{error},
\end{align}

\noindent where showing the error is small as $N\to\infty$ will require some work. Looking at Definition \ref{def:chebcoef}, the sum above looks like a Riemann sum, so if we assume we can approximate it by the corresponding integral, we have 

\begin{align}
\notag\int_{-1}^1 X_N(y)f'(y)dy&=\int_0^1 f'(G^{-1}(t))\frac{1}{\sigma(G^{-1}(t))}X_N(G^{-1}(t))dt\\
&=\sum_{j=1}^N f(\lambda_j)-N\int_{\gamma_{j-1}}^{\gamma_j}f(x)\sigma(x)dx+\mathrm{error}\\
\notag &=\sum_{j=1}^{N}(f(\lambda_j)-\E f(\lambda_j))+\mathrm{error}.
\end{align}

Again this suggests that 

\begin{equation}
X_N(x)\sim \sum_{k=0}^{\infty}\frac{1}{k+1}\sum_{j=1}^{N}(T_{k+1}(\lambda_j)-\E T_{k+1}(\lambda_j))U_k(x)\sqrt{1-x^{2}}.
\end{equation}

From Theorem \ref{th:johansson} the Fourier-Chebyshev coefficients are asymptotically normally distributed with variance $\rho_{(k+1)^{-1}T_{k+1}}$. Note that using the definition of $s_k$, Lemma \ref{le:var}, and Lemma \ref{le:logdiag}, we have formally (the formal part being not justifying interchanging the order of summation and integration)

\begin{align}
\notag\rho_f&=\frac{1}{4}\sum_{k=0}^\infty\frac{1}{k+1}\frac{2}{\pi}\int_{-1}^1f'(x)\sqrt{1-x^2}U_k(x)dx\frac{2}{\pi}\int_{-1}^1 f'(y)\sqrt{1-y^2}U_k(y)dy\\
&=\frac{1}{\pi^2}\int_{-1}^1\int_{-1}^1 f'(x)f'(y)C(x,y)dxdy.
\end{align}

This suggests that $\pi X_N$ converges to $X$ in some sense, and this calculation is indeed the heart of our argument, but making things precise requires some work.

\vspace{0.3cm}

This also suggests that $X_N$ and $\mathbb{X}_N$ are close to each other in some sense. Indeed, controlling the behavior of $\mathbb{X}_N(x)$ near the edge of the spectrum, one should be able to show that $X_N-\mathbb{X}_N$ tends to zero in $\mathcal{S}_{-\alpha}$, though stronger estimates might be interesting to e.g. compare the maxima of the two fields. While a connection between the fields $X_N$ and $\mathbb{X}_N$ is important for example in \cite{dallaporta,gustavsson}, the author has not run into a statement about the precise connection between the two fields. 

\vspace{0.3cm}

We point out that all of these (asymptotically) log-correlated fields encode the linear statistics related to the Chebyshev polynomials and can be used to calculate linear statistics of reasonably smooth functions.

\section{Proof of Theorem \ref{th:main}}

We will now prove our main theorem. As mentioned before, there are essentially two parts to this - proving convergence on certain cylinder sets and proving tightness. We'll first prove convergence on cylinder sets (without even assuming that $X_N\in \mathcal{S}_{-\alpha}$), and finally tightness where our reasoning is similar to that in \cite{fks}.

\subsection{Convergence on cylinder sets} 
The goal of this part of the proof is to prove the following claim.

\begin{lemma}\label{le:fdd}
Let $(t_0,...,t_K)\in \R^{K+1}$ and as before, let $(Y_k)_{k=0}^K$ be i.i.d. standard Gaussians. Then as $N\to \infty$, 

\begin{equation}
\sum_{k=0}^K t_k s_k(X_N)\stackrel{d}{\to}\frac{1}{\pi}\sum_{k=0}^K t_k \frac{1}{\sqrt{k+1}}Y_k.
\end{equation} 
\end{lemma}

Before proving this, we'll prove a simple result concerning the distances between the points $(\gamma_j)_j$.

\begin{lemma}\label{le:gammalp}
For $1<p\leq 3$,

\begin{equation}
\sum_{j=1}^N |\gamma_j-\gamma_{j-1}|^p\leq C N^{1-p}
\end{equation}

\noindent for some $C>0$.

\end{lemma}

\begin{proof}[Proof of Lemma \ref{le:gammalp}]
We'll first prove the following estimates: 

\begin{equation}\label{eq:gammab1}
-1+\left(\frac{j}{N}\right)^{\frac{2}{3}}\leq \gamma_j\leq -1+2\left(\frac{j}{N}\right)^{\frac{2}{3}}
\end{equation}

\noindent for $j<N/2$ and 

\begin{equation}\label{eq:gammab2}
1-\left(\frac{N-j}{N}\right)^{\frac{2}{3}}\geq \gamma_j\geq 1-2\left(\frac{N-j}{N}\right)^{\frac{2}{3}}
\end{equation}

\noindent for $j\geq N/2$.

\vspace{0.3cm}

Let us begin with the upper bound in \eqref{eq:gammab1} (note that for $j<N/2$, $-1+2(j/N)^{2/3}<1$):

\begin{align}
\notag \int_{-1}^{-1+2\left(\frac{j}{N}\right)^{\frac{2}{3}}}\sigma(x)dx&=\frac{2}{\pi}\int_{-1}^{-1+2\left(\frac{j}{N}\right)^{\frac{2}{3}}}\sqrt{1-x}\sqrt{1+x}dx\\
&\geq \frac{2}{\pi}\sqrt{2-2\left(\frac{j}{N}\right)^{\frac{2}{3}}}\int_{-1}^{-1+2\left(\frac{j}{N}\right)^{\frac{2}{3}}}\sqrt{1+x}dx\\
\notag &\geq \frac{2}{\pi}\sqrt{2-2\left(\frac{1}{2}\right)^{\frac{2}{3}}}\int_{0}^{2\left(\frac{j}{N}\right)^{\frac{2}{3}}}\sqrt{x}dx\\
\notag &=\frac{2}{\pi}\sqrt{2-2\left(\frac{1}{2}\right)^{\frac{2}{3}}}\frac{2}{3} 2^{\frac{3}{2}}\frac{j}{N}.
\end{align}

The numerical factor here is greater than one (approximately $1.03$) so we have the upper bound for $j<N/2$. For the lower bound, note that 

\begin{align}
\notag \int_{-1}^{-1+\left(\frac{j}{N}\right)^{\frac{2}{3}}}\sigma(x)dx&\leq \frac{2\sqrt{2}}{\pi}\int_{-1}^{-1+\left(\frac{j}{N}\right)^{\frac{2}{3}}}\sqrt{1+x}dx\\
&\leq \frac{2\sqrt{2}}{\pi}\frac{2}{3}\frac{j}{N}\\
\notag &\leq \frac{j}{N}.
\end{align}

The argument for $j\geq N/2$ is similar apart for the fact that one uses the identity (following from Definition \ref{def:semic})

\begin{equation}
\int_{\gamma_j}^1\sigma(x)dx=\frac{N-j}{N}.
\end{equation}

The next ingredient of the proof is the remark that for some point $x_j\in[\gamma_{j-1},\gamma_j]$

\begin{equation}
N\sigma(x_j)(\gamma_j-\gamma_{j-1})=N\int_{\gamma_{j-1}}^{\gamma_j}\sigma(x)dx=1
\end{equation}

\noindent so in particular, 

\begin{equation}\label{eq:gammad}
(\gamma_j-\gamma_{j-1})\leq \frac{1}{N}\max\left(\frac{1}{\sigma(\gamma_{j-1})},\frac{1}{\sigma(\gamma_j)}\right).
\end{equation}

We split the sum $\sum_j |\gamma_j-\gamma_{j-1}|^p$ into the term $j=1$, the sums $1< j<N/2$ and $N/2\leq j\leq N-1$ and the term $j=N$. For the term $j=1$ \eqref{eq:gammab1} implies that :

\begin{equation}
|\gamma_1+1|^p\leq 2^p\left(\frac{1}{N}\right)^{\frac{2p}{3}}.
\end{equation}

For the sum over $1<j<N/2$ we have by \eqref{eq:gammad} and \eqref{eq:gammab1}

\begin{align}
\notag\sum_{j=2}^{\lfloor N/2\rfloor}|\gamma_j-\gamma_{j-1}|^p&\leq N^{-p}\sum_{j=2}^{\lfloor N/2\rfloor} \frac{1}{\sigma(\gamma_{j-1})^p}\\
&\leq \frac{2\sqrt{2}}{\pi}N^{-p}\sum_{j=2}^{\lfloor N/2\rfloor}(1+\gamma_{j-1})^{-\frac{p}{2}}\\
\notag &\leq \frac{2^{\frac{3}{2}}}{\pi}N^{-p}\sum_{j=2}^{\lfloor N/2\rfloor}\left(\frac{j-1}{N}\right)^{-\frac{p}{2}}\\
\notag&\leq C N^{1-p}
\end{align}

\noindent for some $C>0$. With a similar argument for $j\geq N/2$, we have 

\begin{equation}
\sum_{j=1}^N|\gamma_j-\gamma_{j-1}|^p\leq C \left(N^{-\frac{2p}{3}}+N^{1-p}\right)=C N^{1-p}\left(1+N^{\frac{p}{3}-1}\right).
\end{equation}

Which gives the claim as $p\leq 3$.

\end{proof}

\begin{proof}[Proof of Lemma \ref{le:fdd}]
We argue as in our heuristic proof, but with a bit more detail. Consider first an arbitrary $f\in C^\infty_c(\R)$ and let us look at the statement of Theorem \ref{th:johansson} a bit more carefully. We Taylor expand $f(\lambda_j)$ around $\gamma_j$ to second order as well as write the integral as a sum of integrals over $[\gamma_{j-1},\gamma_j]$ and on each of these, we Taylor expand $f(x)$ in the integral around $\gamma_j$ to second order. We thus obtain (using Definition \ref{def:semic})

\begin{align}
\notag \sum_{j=1}^N& f(\lambda_j)-N\int_{-1}^1f(x)\sigma(x)dx\\
\notag &=\sum_{j=1}^N\left(f(\gamma_j)+f'(\gamma_j)(\lambda_j-\gamma_j)+\frac{1}{2}f''(t_j)(\lambda_j-\gamma_j)^2\right)\\
\notag &\quad -\sum_{j=1}^N N\int_{\gamma_{j-1}}^{\gamma_j}\left(f(\gamma_j)+f'(\gamma_j)(x-\gamma_j)+\frac{1}{2}f''(u_{j,x})(x-\gamma_j)^2\right)\sigma(x)dx\\
&=\sum_{j=1}^Nf'(\gamma_j)\left(\gamma_j-N\int_{\gamma_{j-1}}^{\gamma_{j}} x\sigma(x)dx\right)+\mathcal{E}_N(f)\\
\notag &=\frac{1}{N}\sum_{j=1}^N f'(\gamma_j)\frac{1}{\sigma(\gamma_j)}X_N(\gamma_j)+\mathcal{E}_N(f),
\end{align}

\noindent where $t_j$ is a point between $\lambda_j$ and $\gamma_j$, $u_{j,x}$ is a point between $x$ and $\gamma_j$, when $x\in[\gamma_{j-1},\gamma_j]$, and 

\begin{equation}\label{eq:en}
\mathcal{E}_N(f)=\sum_{j=1}^N\left(\frac{1}{2}f''(t_j)(\lambda_j-\gamma_j)^2+\frac{N}{2}\int_{\gamma_{j-1}}^{\gamma_j}f''(u_{j,x})(x-\gamma_j)^2\sigma(x)dx\right).
\end{equation}

Note that as $f\in C^\infty_c(\R)$, $||f''||_{\infty}$ is finite so (using Definition \ref{def:semic} for the second summand)

\begin{equation}\label{eq:ebound}
|\mathcal{E}_N(f)|\leq C\sum_{j=1}^N \left(|\lambda_j-\gamma_j|^2+|\gamma_j-\gamma_{j-1}|^2\right).
\end{equation}

\noindent for some $C$ depending only on $f$. By Proposition \ref{prop:dallaporta} and Lemma \ref{le:gammalp}, we see that as $N\to\infty$, 

\begin{equation}
\E|\mathcal{E}_N(f)|\to 0
\end{equation}

\noindent so $\mathcal{E}_N(f)\stackrel{d}{\to} 0$ as $N\to\infty$.

\vspace{0.3cm}

Let us now return to the term

\begin{equation}
\frac{1}{N}\sum_{j=1}^Nf'(\gamma_j)\frac{1}{\sigma(\gamma_j)}X_N(\gamma_j).
\end{equation}

As in our heuristic argument, we want to see this as a Riemann sum (though the integrand depends on $N$ through $X_N$) so let us estimate its distance from the corresponding integral. We have 

\begin{align}\label{eq:dn}
\Delta_N(f):&=\notag \frac{1}{N}\sum_{j=1}^N\frac{f'(\gamma_j)}{\sigma(\gamma_j)}X_N(\gamma_j)-\int_0^1 \frac{f'(G^{-1}(t))}{\sigma(G^{-1}(t))}X_N(G^{-1}(t))dt\\
\notag &=\sum_{j=1}^N\frac{X_N(\gamma_j)}{N}\left( \frac{f'(\gamma_j)}{\sigma(\gamma_j)}-N\int_{\frac{j-1}{N}}^{\frac{j}{N}}\frac{f'(G^{-1}(t))}{\sigma(G^{-1}(t))}dt\right)\\
&=\sum_{j=1}^N\left(\lambda_j-N\int_{\gamma_{j-1}}^{\gamma_j}x\sigma(x)dx\right)\left( f'(\gamma_j)-N\sigma(\gamma_j)\int_{\gamma_{j-1}}^{\gamma_j}f'(y)dy\right)\\
\notag &=\sum_{j=1}^N\left(\lambda_j-N\int_{\gamma_{j-1}}^{\gamma_j}x\sigma(x)dx\right) f'(\gamma_j)\left(1-N\sigma(\gamma_j)(\gamma_j-\gamma_{j-1})\right)\\
\notag &\quad -\sum_{j=1}^N \left(\lambda_j-N\int_{\gamma_{j-1}}^{\gamma_j}x\sigma(x)dx\right) N\sigma(\gamma_j)\int_{\gamma_{j-1}}^{\gamma_j}f''(w_j)(y-\gamma_j)dy\\
\notag &=:\Delta_N^{(1)}(f)+\Delta_N^{(2)}(f).
\end{align}

Let us first note that for $\Delta_{N}^{(1)}(f)$ we have the bound

\begin{align}
\notag\E|\Delta_{N}^{(1)}(f)|&\leq \sup_{x\in[-1,1]}|f'(x)|\sum_{j=1}^N\left(\E|\lambda_j-\gamma_j|+\left|\gamma_j-N\int_{\gamma_{j-1}}^{\gamma_j}x\sigma(x)dx\right|\right)\\
&\qquad\times N\int_{\gamma_{j-1}}^{\gamma_j}|\sigma(x)-\sigma(\gamma_j)|dx.
\end{align}

\noindent We begin by noting that Proposition \ref{prop:dallaporta}, part 4 implies that 

\begin{equation}\label{eq:dn1}
\sup_j \E|\lambda_j-\gamma_j|\leq \sqrt{\sup_j \E|\lambda_j-\gamma_j|^2}\leq C\sqrt{\frac{\log N}{N}}.
\end{equation}

Similarly, Lemma \ref{le:gammalp} (with $p=3$) implies that 

\begin{equation}\label{eq:dn2}
\sup_j \left|\gamma_j-N\int_{\gamma_{j-1}}^{\gamma_j}x\sigma(x)dx\right|\leq \sup_j |\gamma_j-\gamma_{j-1}|\leq C N^{-\frac{2}{3}}.
\end{equation}

For the last term we have (again using Lemma \ref{le:gammalp} with $p=3$)

\begin{align}\label{eq:dn3}
\notag N\int_{\gamma_{j-1}}^{\gamma_j}|\sigma(x)-\sigma(\gamma_j)|dx&=N\frac{2}{\pi}\int_{\gamma_{j-1}}^{\gamma_j}\frac{|x-\gamma_j|(x+\gamma_j)}{\sqrt{1-x^2}+\sqrt{1-\gamma_j^2}}dx\\
&\leq CN (\gamma_j-\gamma_{j-1})\int_{\gamma_{j-1}}^{\gamma_j}\frac{1}{\sqrt{1-x^2}}dx\\
\notag &\leq C N^{1/3}\int_{\gamma_{j-1}}^{\gamma_j}\frac{1}{\sqrt{1-x^2}}dx.
\end{align}

Combining \eqref{eq:dn1}, \eqref{eq:dn2}, and \eqref{eq:dn3} we find (for some constant independent of $f$)

\begin{equation}\label{eq:delta1}
\E|\Delta_N^{(1)}(f)|\leq C \sup_{x\in[-1,1]}|f'(x)|N^{-\frac{1}{6}}\sqrt{\log N}\int_{-1}^1\frac{1}{\sqrt{1-x^2}}dx\to 0 
\end{equation}

\noindent as $N\to \infty$. We keep that $f'$ term here instead of absorbing it into the constant for an argument we'll need when proving tightness. We conclude that $\Delta_N^{(1)}(f)\stackrel{d}{\to} 0$ as $N\to \infty$.

\vspace{0.3cm}

For $\Delta_N^{(2)}(f)$ we find (with a slightly simpler argument - using Lemma \ref{le:gammalp} with $p=2$)

\begin{align}\label{eq:delta2}
\notag \E|\Delta_N^{(2)}(f)|&\leq C\sup_{x\in[-1,1]}|f''(x)|\sqrt{\frac{\log N}{N}}N\sum_{j=1}^N\sigma(\gamma_j)(\gamma_j-\gamma_{j-1})^2\\
&\leq C \sup_{x\in[-1,1]}|f''(x)|\sqrt{N\log N}\sum_{j=1}^N |\gamma_j-\gamma_{j-1}|^2\\
&\notag \leq C \sup_{x\in[-1,1]}|f''(x)|\sqrt{\frac{\log N}{N}}
\end{align}

\noindent for some constant independent of $f$. 

\vspace{0.3cm}

We conclude that $\Delta_N(f)\stackrel{d}{\to} 0$ as $N\to \infty$ so by Slutsky's theorem and Theorem \ref{th:johansson}, we find that as $N\to \infty$

\begin{equation}
\int_0^1 \frac{f'(G^{-1}(t))}{\sigma(G^{-1}(t))}X_N(G^{-1}(t))dt=\int_{-1}^1 f'(y)X_N(y)dy
\end{equation}

\noindent converges in law to a centered Gaussian random variable with variance (see Lemma \ref{le:var})

\begin{equation}
\rho_f=\frac{1}{4}\sum_{k=0}^\infty\frac{1}{k+1}s_k(f'(x)\sqrt{1-x^2})^2.
\end{equation}

Let us now choose our function $f\in C^\infty_c(\R)$ in a suitable way. Let us take $f'$ so that for $x\in [-1,1]$ we have 

\begin{equation}\label{eq:der}
f'(x)=\frac{2}{\pi}\sum_{k=0}^K t_k U_k(x).
\end{equation}

Thus 

\begin{equation}
\int_{-1}^1 X_N(y)f'(y)dy=\sum_{k=0}^K t_k s_k(X_N).
\end{equation}

Let us then choose $f$ to be any $C^\infty_c(\R)$ function whose derivative on $[-1,1]$ is given by \eqref{eq:der}. We note that for such a function

\begin{equation}
s_k(f'(x)\sqrt{1-x^2})=\sum_{j=0}^K t_j\frac{4}{\pi^2}\int_{-1}^1 U_j(x)U_k(x)\sqrt{1-x^2}dx=\frac{2}{\pi}t_k\mathbf{1}(k\leq K)
\end{equation}

\noindent and 

\begin{equation}
\rho_f=\frac{1}{\pi^2}\sum_{k=0}^K \frac{1}{k+1}t_k^2
\end{equation}

\noindent which is precisely the variance of the normal random variable

\begin{equation}
\frac{1}{\pi}\sum_{k=0}^K \frac{1}{\sqrt{k+1}}t_k Y_k
\end{equation}

\noindent so we are done.

\end{proof}

\subsection{Tightness}

Let us first outline how we'll prove tightness (following \cite{hko}). The basic result we'll use is that one can prove that the unit ball of $\mathcal{S}_{-\alpha'}$ is a compact subset of $\mathcal{S}_{-\alpha}$ for $\alpha>\alpha'$. Thus an application of Markov's inequality will prove tightness if one can prove that $\E||X_N||_{-\alpha'}^p<\infty$ for some $p>0$. In \cite{hko,fks} the exponent $p$ is taken to be 2, but we'll find it more convenient to take $p=1$. $p=2$ would be a natural choice if one had good bounds for the covariance of $X_N$, but to the author's knowledge these don't exist.

To estimate $\E||X_N||_{-\alpha'}$, we note that 

\begin{equation}\label{eq:tight}
||X_N||_{-\alpha'}=\sqrt{\sum_{k=0}^\infty s_k(X_N)^2 (1+k^2)^{-\alpha'}}\leq \sum_{k=0}^\infty |s_k(X_N)|(1+k^2)^{-\frac{\alpha'}{2}}
\end{equation}

\noindent so we'll need polynomial (in $k$) growth bounds on $\E|s_k(X_N)|$ to get boundedness for some $\alpha'>0$. In fact, our argument will consist of two parts and will be slightly roundabout. Our first bound will be useful only when $k$ is large enough (comparable to some power of $\log N$) and will be a simple application of Proposition \ref{prop:dallaporta}.

\begin{lemma}\label{le:bigk}
There exists a constant $C>0$ such that 

\begin{equation}
\E|s_k(X_N)|^2\leq C(k+1)^2\log N.
\end{equation}

\end{lemma}

\begin{proof}
Writing out the definition of $s_k(X_N)$ and using a simple Cauchy-Schwarz and other naive estimates we find

\begin{align}
\notag \E|s_k(X_N)|^2&=\sum_{j=1}^N\sum_{l=1}^N \frac{4N^2}{\pi^2}\sigma(\gamma_j)\sigma(\gamma_l)\int_{\gamma_{j-1}}^{\gamma_j}U_k(x)dx \int_{\gamma_{l-1}}^{\gamma_l}U_k(x)dx\\
&\qquad\times \E\left[\left(\lambda_j-N\int_{\gamma_{j-1}}^{\gamma_j}x\sigma(x)dx\right)\left(\lambda_l-N\int_{\gamma_{l-1}}^{\gamma_l}x\sigma(x)dx\right)\right]\\
\notag &\leq C N^2\left(\sup_{x\in[-1,1]}|U_k(x)|\right)^2\\
\notag &\qquad \times\left(\sum_{j=1}^N\sigma(\gamma_j)\sqrt{\E\left|\lambda_j-N\int_{\gamma_{j-1}}^{\gamma_j}x\sigma(x)dx\right|^2}(\gamma_j-\gamma_{j-1})\right)^2.
\end{align}

Note that as $U_k(x)=\frac{1}{k+1}T_{k+1}'(x)$, Proposition \ref{prop:markov} implies that 

\begin{equation}
\sup_{x\in[-1,1]}|U_k(x)|\leq k+1.
\end{equation}

For the expectation, we remark that

\begin{equation}
\E\left|\lambda_j-N\int_{\gamma_{j-1}}^{\gamma_j}x\sigma(x)dx\right|^2\leq 2\E|\lambda_j-\gamma_j|^2+2(\gamma_j-\gamma_{j-1})^2.
\end{equation}

Let us now split the sum into the different sectors indicated by Proposition \ref{prop:dallaporta}. Consider first $j\leq K\log N$. We have by \eqref{eq:gammab1},  Lemma \ref{le:gammalp}, and Proposition \ref{prop:dallaporta} part 3

\begin{align}
\notag \sum_{j=1}^{\lfloor K\log N\rfloor}&\sigma(\gamma_j)\sqrt{\E|\lambda_j-\gamma_j|^2+(\gamma_j-\gamma_{j-1})^2}(\gamma_j-\gamma_{j-1})\\
\qquad &\leq C\sum_{j=1}^{\lfloor K\log N\rfloor}\left(\frac{j}{N}\right)^{\frac{1}{3}}\sqrt{\frac{1}{N}} N^{-\frac{2}{3}}\\
\notag &\leq C N^{-\frac{3}{2}}(\log N)^{\frac{4}{3}}.
\end{align}

We get the same bound for $N-j\leq K\log N$. For $K\log N<j\leq \delta N$, we need a little bit more care for estimating $\gamma_j-\gamma_{j-1}$. As we have already argued, we have for some $x_j\in[\gamma_{j-1},\gamma_j]$

\begin{equation}
N\sigma(x_j)(\gamma_j-\gamma_{j-1})=1
\end{equation}

\noindent which implies that for $K\log N<j\leq \delta N$, there exists a universal constant $C$ such that 

\begin{equation}
(\gamma_j-\gamma_{j-1})\leq \frac{1}{N\sigma(\gamma_{j-1})}\leq C N^{-\frac{2}{3}} j^{-\frac{1}{3}}.
\end{equation}

Combining this with part 2 of Proposition \ref{prop:dallaporta} we find

\begin{align}
\notag \sum_{j=\lfloor K\log N\rfloor +1}^{\lfloor \delta N\rfloor }&\sigma(\gamma_j)\sqrt{\E|\lambda_j-\gamma_j|^2+(\gamma_j-\gamma_{j-1})^2}(\gamma_j-\gamma_{j-1})\\
\qquad &\leq C\sum_{j=\lfloor K\log N\rfloor +1}^{\lfloor \delta N\rfloor }\left(\frac{j}{N}\right)^{\frac{1}{3}}\sqrt{\log j N^{-\frac{4}{3}}j^{-\frac{2}{3}}+N^{-\frac{4}{3}}j^{-\frac{2}{3}}} N^{-\frac{2}{3}}j^{-\frac{1}{3}}\\
\notag &\leq  CN^{-\frac{5}{3}}\sum_{j=\lfloor K\log N\rfloor +1}^{\lfloor \delta N\rfloor }\sqrt{\log j}j^{-\frac{1}{3}}\\
\notag &\leq C\sqrt{\log N} N^{-1}.
\end{align}

A similar bound holds for $K\log N\leq N-j\leq \delta N$. Finally for $\delta N<j<(1-\delta)N$ we have using part 1 of Proposition \ref{prop:dallaporta} (and noting that $\gamma_j-\gamma_{j-1}\leq C N^{-1}$)

\begin{align}
\notag \sum_{N\delta <j<N(1-\delta)}&\sigma(\gamma_j)\sqrt{\E|\lambda_j-\gamma_j|^2+(\gamma_j-\gamma_{j-1})^2}(\gamma_j-\gamma_{j-1})\\
&\leq C\sqrt{\log N} N^{-1}.
\end{align}

Putting the different estimates together yields our claim.

\end{proof}

For a small $k$ estimate of $\E|s_k(X_N)|$, we go back to our argument for convergence on cylinder sets, and make use of variance bounds for linear statistics of the GUE. Let us write $f_k=(k+1)^{-1}T_{k+1}$. Recall that we had 

\begin{align}
\frac{\pi}{2}\notag s_k(X_N)&=\int_{-1}^1 X_N(x)f_k'(x)dx\\
&=\sum_{j=1}^N f(\lambda_j)-N\int_{-1}^1 f_k(x)\sigma(x)dx-\Delta_N(f_k)-\mathcal{E}_N(f_k)
\end{align}

Let us begin our analysis for small $k$ by estimating $\Delta_N(f_k)$ first. 

\begin{lemma}\label{le:delta}
There exists an $\epsilon>0$ and a $C>0$ such that for $k\leq \sqrt{\log N}$

\begin{equation}
\E|\Delta_N(f_k)|\leq C N^{-\epsilon}.
\end{equation}

\end{lemma}

\begin{proof}
We already found (see \eqref{eq:delta1} and \eqref{eq:delta2} that 

\begin{equation}
 \E\left|\Delta_N(f_k)\right|\leq C\sup_{x\in[-1,1]}|f_k'(x)|N^{-1/6}\sqrt{\log N}+C\sup_{x\in[-1,1]}|f_k''(x)|\sqrt{\frac{\log N}{N}}.
\end{equation}

Now making use of Proposition \ref{prop:markov}, we have 

\begin{equation}
\frac{1}{k+1}\sup_{x\in[-1,1]}|T_{k+1}'(x)|\leq k+1 
\end{equation}

\noindent and 

\begin{equation}
\frac{1}{k+1}\sup_{x\in[-1,1]}|T_{k+1}''(x)|\leq (k+1)\frac{k^2+2k}{3}.
\end{equation}

We conclude (as $k\leq \sqrt{\log N}$) that for some constants $C>0$ and $\epsilon>0$.  

\begin{equation}
\E|\Delta_N(f_k)|\leq C N^{-\epsilon}.
\end{equation}

\end{proof}

Next we prove a bound  for $\E|\mathcal{E}_N(f_k)|$, though the proof is more involved.

\begin{lemma}\label{le:enfk}
There exist constants $C>0$ and $\epsilon>0$ such that for $k\leq \sqrt{\log N}$,

\begin{equation}
\E|\mathcal{E}_N(f_k)|\leq C N^{-\epsilon}.
\end{equation}

\end{lemma}

\begin{proof}
From \eqref{eq:en} we see that for some point $t_j$ between $\lambda_j$ and $\gamma_j$, 

\begin{equation}
|\mathcal{E}_N(f_k)|\leq \frac{1}{2}\sum_{j=1}^N |f_k''(t_j)|(\lambda_j-\gamma_j)^2+\frac{1}{2}\sup_{x\in[-1,1]}|f_k''(x)|\sum_{j=1}^N(\gamma_j-\gamma_{j-1})^2.
\end{equation}

Then by Lemma \ref{le:gammalp} and our above bound for $|f_k''(x)|$ we see that the second term is bounded by $Ck^3 N^{-1}$ for some constant $C$. Let us thus estimate the first term. For this, we use the following bound 

\begin{align}\label{eq:e9}
\notag \sum_{j=1}^N |f_k''(t_j)|(\lambda_j-\gamma_j)^2&\leq \sup_{x\in[-1,1]}|f_k''(x)|\sum_{j=1}^N (\lambda_j-\gamma_j)^2\\
&\qquad +\sum_{j=1}^N |f_k''(t_j)|(\lambda_j-\gamma_j)^2\mathbf{1}(|\lambda_j|>1).
\end{align}

By part 4 of Proposition \ref{prop:dallaporta}, the expectation of the first term here can be bounded by $C k^3 \log N/N\leq C N^{-\epsilon}$ for some $\epsilon>0$. To estimate the last term, we first note that as $t_j$ is between $\gamma_j$ and $\lambda_j$, Lemma \ref{le:chebyder} (and using the fact that $x\mapsto|T_n^{(k)}(x)|$ is even for all $n,k$) as well as Proposition \ref{prop:markov} imply that 

\begin{equation}\label{eq:e8}
|f_k''(t_j)|\mathbf{1}(|\lambda_j|>1)\leq |f_k''(\lambda_j)|\mathbf{1}(|\lambda_j|>1).
\end{equation}

Then by Cauchy-Schwarz we find

\begin{align}\label{eq:e7}
\notag \E\sum_{j=1}^N &|f_k''(\lambda_j)|(\lambda_j-\gamma_j)^2\mathbf{1}(|\lambda_j|>1)\\
&\leq \sum_{j=1}^N\sqrt{\E|f_k''(\lambda_j)|^2\mathbf{1}(|\lambda|_j>1)}\sqrt{\E|\lambda_j-\gamma_j|^4}\\
\notag &\leq \sqrt{\sum_{j=1}^N\E|f_k''(\lambda_j)|^2\mathbf{1}(|\lambda_j|>1)}\sum_{j=1}^N \sqrt{\E|\lambda_j-\gamma_j|^4}.
\end{align}

Now from Proposition \ref{prop:dallaporta} we find for some $C>0$

\begin{align}\label{eq:e6}
\notag \sum_{j=1}^N \sqrt{\E|\lambda_j-\gamma_j|^4}&\leq 2C\sum_{j<K\log N}\frac{1}{N}+2C\sum_{K\log N\leq j\leq \delta N}\frac{\log j}{N^{4/3} j^{2/3}}\\
&\qquad+2C\sum_{\delta N<j\leq N/2}\frac{\log N}{N^2}\\
\notag &=C\frac{\log N}{N}.
\end{align}

Thus we are left with estimating 

\begin{equation}\label{eq:e5}
\sum_{j=1}^N \E|f_k''(\lambda_j)|^2 \mathbf{1}(|\lambda_j|>1)=N\int_{|x|>1}|f_k''(x)|^2 \rho_N(x)dx.
\end{equation}

Now combining Lemma \ref{le:tder} and Lemma \ref{le:tbound} (as well as using the fact that we have even functions), we find

\begin{equation}\label{eq:e4}
N\int_{|x|>1}|f_k''(x)|^2 \rho_N(x)dx\leq 2N(k+1)^6\int_1^\infty e^{4k\sqrt{x-1}}\rho_N(x)dx.
\end{equation}

We now wish to make use of Proposition \ref{prop:density}. To do this, let $\delta_N=N^{-1/2}$ (we could actually get away here with taking $\delta_N$ some small positive constant $\delta$, but we'll need this argument later on). We'll use part 1 of Proposition \ref{prop:density} for estimating the above integral from $1$ to $1+\delta_N$. So more precisely,

\begin{align}\label{eq:ints}
\notag &\int_{1}^{1+\delta_N}e^{4k\sqrt{x-1}}N\rho_N(x)dx\\
\notag &\quad =\int_{1}^{1+\delta_N}e^{4k\sqrt{x-1}}\left(\frac{\Phi'(x)}{4\Phi(x)}-\frac{\gamma'(x)}{\gamma(x)}\right) 2 \mathrm{Ai}(N^{2/3}\Phi(x))\mathrm{Ai}'(N^{2/3}\Phi(x))dx\\
&\qquad + \int_1^{1+\delta_N}e^{4k\sqrt{x-1}}N^{2/3}\Phi'(x)\\
\notag &\qquad \qquad \times\left(\mathrm{Ai}'(N^{2/3}\Phi(x))^2-N^{2/3}\Phi(x)\mathrm{Ai}(N^{2/3}\Phi(x))^2\right)dx\\
&\notag \qquad +\int_1^{1+\delta_N}e^{4k\sqrt{x-1}}\mathcal{O}\left(\frac{1}{N\sqrt{x-1}}\right)dx.
\end{align}

For $x\geq 0$, $\mathrm{Ai}(x)$ and $\mathrm{Ai}'(x)$ are bounded. Moreover, (as mentioned in \cite{fks})

\begin{equation}
\frac{\Phi'(x)}{4\Phi(x)}-\frac{\gamma'(x)}{\gamma(x)}
\end{equation}

\noindent is bounded near $x=1$ so we see that the first integral can be bounded by 

\begin{equation}
C\int_1^{1+\delta_N} e^{4k\sqrt{x-1}}dx\leq C e^{4k \sqrt{\delta_N}}.
\end{equation}

For the second integral in \eqref{eq:ints}, we estimate $e^{4k\sqrt{x-1}}$ upwards by $e^{4k\sqrt{\delta_N}}$ and then as in \cite{fks}, the integral can be evaluated by the substitution $u=N^{2/3}\Phi'(x)$ and as the Airy function is bounded, one finds

\begin{align}
&\int_1^{1+\delta_N}e^{4k\sqrt{x-1}}N^{2/3}\Phi'(x)\left(\mathrm{Ai}'(N^{2/3}\Phi(x))^2-N^{2/3}\Phi(x)\mathrm{Ai}(N^{2/3}\Phi(x))^2\right)dx\\
&\leq C e^{4k\sqrt{\delta_N}}.
\end{align}

Finally for the last part, one notes that 

\begin{equation}
\frac{1}{N}\int_1^{1+\delta_N} e^{4k\sqrt{x-1}}\frac{1}{\sqrt{x-1}}dx=\frac{1}{N}\int_0^{\sqrt{\delta_N}}e^{4ky}dy\leq \frac{1}{Nk}e^{4k\sqrt{\delta_N}}.
\end{equation}

Putting things together, we find 

\begin{equation}\label{eq:e3}
\int_1^{1+\delta_N}e^{4k\sqrt{x-1}}N\rho_N(x)dx\leq C e^{4k\sqrt{\delta_N}}
\end{equation}

\noindent which is bounded for $k\leq \sqrt{\log N}$.

\vspace{0.3cm}

Fix now some $\delta>0$ independent of $N$ and consider the integral over $[1+\delta_N,1+\delta]$. Here we want to use part 2 of Proposition \ref{prop:density}. First of all, we have

\begin{align}
\notag\int_{1+\delta_N}^{1+\delta}e^{4k\sqrt{x-1}}N\rho_N(x)dx&=\int_{\delta_N}^{\delta} e^{4k\sqrt{x}}N\rho_N(1+x)dx\\
&=N^{1/3}\int_{N^{2/3}\delta_N}^{\delta N^{2/3}} e^{4k\sqrt{s} N^{-1/3}}\rho_N(1+s N^{-2/3})ds.
\end{align}

Now as $N^{2/3}\delta_N\to \infty$ as $N\to\infty$, we find by part 2 of Proposition \ref{prop:density} (and recalling $k\leq \sqrt{\log N}$)

\begin{align}\label{eq:e2}
\notag \int_{1+\delta_N}^{1+\delta}e^{4k\sqrt{x-1}}N\rho_N(x)dx&\leq \int_{N^{2/3}\delta_N}^{\delta N^{2/3}} e^{4k\sqrt{s} N^{-1/3}}\frac{1}{B s}e^{-b s^{3/2}}ds\\
&\leq C N^{-1/3}\int_{N^{1/3}}^{\delta N^{2/3}}e^{\sqrt{s}(4k N^{-1/3}-bs)}ds\\
&\notag \leq C N^{-1/3}\int_{N^{1/3}}^{\delta N^{2/3}} e^{-b'\sqrt{s}}ds\\
&\notag \to 0
\end{align}

\noindent as $N\to \infty$. 

\vspace{0.3cm}

Finally by part 3 of Proposition \ref{prop:density}, we find

\begin{equation}\label{eq:e1}
\int_{1+\delta}^{\infty}e^{4k\sqrt{x-1}}N\rho_N(x)dx\leq N\int_{1+\delta}^{\infty}e^{4k\sqrt{x-1}}C e^{-cN x^{2}}dx\to 0
\end{equation}

\noindent as $N\to \infty$. Collecting everything (i.e. \eqref{eq:e9}, \eqref{eq:e8}, \eqref{eq:e7}, \eqref{eq:e6}, \eqref{eq:e5}, \eqref{eq:e4}, \eqref{eq:e3}, \eqref{eq:e2} and \eqref{eq:e1})

\begin{equation}
\E|\mathcal{E}_N(f_k)|\leq C N^{-\epsilon}
\end{equation}

\noindent for some $C,\epsilon>0$ independent of $k$ and $N$.

\end{proof}

To estimate $\sum_{j=1}^{N}f_k(\lambda_j)-N\int_{-1}^{1}f_k(x)\sigma(x)dx$, we'll want to replace the integral by the expectation of the sum. While this is not completely trivial as $f_k$ can depend on $N$, and in particular, have an unbounded derivative, it is still a simple proof we'll present here, though similar things might exist in the literature already.

\begin{lemma}\label{le:els}
For $k\leq \sqrt{\log N}$,

\begin{equation}
N\int_{-1}^{1}f_k(x)\sigma(x)dx-\sum_{j=1}^{N}\E f_k(\lambda_j)\to 0
\end{equation}

\noindent as $N\to \infty$. 
\end{lemma}

\begin{proof}
Let us begin by writing

\begin{equation}
f_k(x)=\frac{1}{k+1}T_{k+1}(x)=\frac{1}{k+1}\sum_{j=0}^{k+1}\frac{T_{k+1}^{(j)}(0)}{j!}x^{j}.
\end{equation}

We then have (see \cite[Section 2.1.1]{agz})

\begin{equation}
N\int_{-1}^{1}f(x)\sigma(x)dx=\frac{N}{k+1}\sum_{j=0}^{\lfloor\frac{k+1}{2}\rfloor}\frac{T_{k+1}^{(2j)}(0)}{(2j)!}4^{-j}\frac{1}{j+1}\binom {2j}{j}.
\end{equation}

On the other hand, we have (for more information, see e.g. \cite[Section 3.3.2]{agz})

\begin{equation}
\sum_{j=1}^{N}\E f(\lambda_j)=\frac{N}{k+1}\sum_{j=0}^{\lfloor\frac{k+1}{2}\rfloor}\frac{T_{k+1}^{(2j)}(0)}{(2j)!}b_j^{(N)}4^{-j}\frac{1}{j+1}\binom {2j}{j},
\end{equation}

\noindent where $(b_j^{(N)})_j$ satisfy the Harer-Zagier recursion:

\begin{equation}
b_{k+1}^{(N)}=b_k^{(N)}+\frac{k(k+1)}{N^{2}}b_{k-1}^{(N)}
\end{equation}

\noindent with initial data $b_{-1}^{(N)}=0$ and $b_0^{(N)}=1$.

\vspace{0.3cm}

In \cite{agz} it is proven (see \cite[equation (3.3.13)]{agz} and around it) that there exists a constant $c>0$ such that for $j\geq 1$

\begin{equation}
1\leq b_j^{(N)}\leq e^{c\frac{j^{3}}{N^{2}}}.
\end{equation}

Also we note that by Stirling's approximation, one can check that

\begin{equation}
\binom{2j}{j}\frac{1}{j+1}4^{-j}\leq C (j+1)^{-3/2}
\end{equation}

\noindent for some $C> 0$. Putting things together and using Proposition \ref{prop:markov} and Lemma \ref{le:tbound}

\begin{align}
\notag \Bigg|\E\sum_{j=1}^{N}&f(\lambda_j)-N\int_{-1}^{1}f(x)\sigma(x)dx\Bigg|\\
\notag  &\leq C\frac{N}{k+1}\sum_{j=0}^{\left\lfloor \frac{k+1}{2}\right\rfloor}\frac{T_{k+1}^{(2j)}(1)}{(2j)!}(j+1)^{-3/2}(b_j^{(N)}-1)\\
&\leq C\frac{N}{k+1}\sum_{j=0}^{\left\lfloor\frac{k+1}{2}\right\rfloor}\frac{T_{k+1}^{(2j)}(1)}{(2j)!}(j+1)^{-3/2}\frac{j^{3}}{N^{2}}\\
\notag &\leq C\frac{\sqrt{k+1}}{N}T_{k+1}(2)\\
&\notag \leq C\frac{\sqrt{k+1}}{N}e^{2 k}.
\end{align}

For $k\leq \sqrt{\log N}$, this tends to zero as $N\to \infty$.
\end{proof}

Our last step in proving a bound for $\E|s_k(X_N)|$ is estimating the centered linear statistic. Our argument is very similar to \cite[Section 5]{fks}.

\begin{lemma}\label{le:ls}
There exists a constant $C>0$ such that for $k\leq \sqrt{\log N}$ and all $N$, 

\begin{equation}
\E\left|\sum_{j=1}^{N}(f_k(\lambda_j)-\E f_k(\lambda_j))\right|\leq C.
\end{equation}
\end{lemma}

\begin{proof}
By the determinantal structure of the eigenvalues of the GUE (for details, see \cite{ps})

\begin{align}\label{eq:lsint}
\notag \left(\E\left|\sum_{j=1}^{N}(f_k(\lambda_j)-\E f_k(\lambda_j))\right|\right)^{2}&\leq \E\left|\sum_{j=1}^{N}(f_k(\lambda_j)-\E f_k(\lambda_j))\right|^{2}\\
&=\frac{1}{8}\int_{\R^{2}}(f_k(x)-f_k(y))^{2}K_N(x,y)^{2}dxdy,
\end{align}

\noindent where $K_N$ is the correlation kernel of the GUE which can be written as 

\begin{equation}
K_N(x,y)=\frac{\psi_N^{(N)}(x)\psi_{N-1}^{(N)}(y)-\psi_{N-1}^{(N)}(x)\psi_N^{(N)}(y)}{x-y},
\end{equation}

\noindent where $\psi_l^{(N)}(x)=e^{-Nx^{2}}P_l^{(N)}(x)$ with $P_l^{(N)}$ being a Hermite polynomial of degree $l$ normalized such that $(\psi_l^{(N)})_l$ is orthonormal in $L^{2}(\R,dx)$.

\vspace{0.3cm}

In \cite[Section 5.2]{fks} it is argued that 

\begin{equation}
\int_{[-1,1]^{2}}(f_k(x)-f_k(y))^{2}K_N(x,y)^{2}dxdy\leq C
\end{equation}

\noindent for some constant $C$ independent of $k$ and $N$. Thus we need to only estimate the integral in \eqref{eq:lsint} outside of this square. Using symmetry under interchanging $x$ and $y$ we find 

\begin{align}
\notag \int_{\R^{2}\setminus [-1,1]^{2}}&(f_k(x)-f_k(y))^{2}K_N(x,y)^{2}dxdy\\
&\leq 2\int_{|x|\geq 1}\int_{|y|\leq |x|}(f_k(x)-f_k(y))^{2}K_N(x,y)^{2}dxdy.
\end{align}

Then by Lemma \ref{le:chebyder} and Proposition \ref{prop:markov}, (in the special case  where in the notation of the lemma and proposition $k=0$ - i.e. we don't differentiate at all) $|f_k(x)|\geq |f_k(y)|$ so we have 

\begin{align}
\notag \int_{\R^{2}\setminus [-1,1]^{2}}&(f_k(x)-f_k(y))^{2}K_N(x,y)^{2}dxdy\\
&\leq 8\int_{|x|\geq 1}\int_{|y|\leq |x|}f_k(x)^{2}K_N(x,y)^{2}dxdy\\
&\notag \leq 8\int_{|x|\geq 1}f_k(x)^{2}\int_\R K_N(x,y)^{2}dxdy\\
&\notag = 8\int_{|x|\geq 1}f_k(x)^{2}N\rho_N(x)dx,
\end{align}

\noindent where we used the reproducing property of the kernel. This integral can be bounded (Lemma \ref{le:tbound}) by the integral 

\begin{equation}
\frac{16}{k^{2}}\int_1^{\infty}e^{4k\sqrt{x-1}}N\rho_N(x)dx
\end{equation}

\noindent which we've estimated already and found that the integral is bounded for $k\leq \sqrt{\log N}$ so we conclude that there exists a constant $C>0$ such that for all $N$,

\begin{equation}
\sup_{k\leq \sqrt{\log N}}\E\left|\sum_{j=1}^{N}(f(\lambda_j)-\E f(\lambda_j))\right|\leq C.
\end{equation}

\end{proof}

Combining Lemma \ref{le:delta}, Lemma \ref{le:enfk}, Lemma \ref{le:els}, and Lemma \ref{le:ls}, we find

\begin{lemma}\label{le:smallk}
There exists a constant $C>0$ such that for $k\leq \sqrt{\log N}$, 

\begin{equation}
\E|s_k(X_N)|\leq C
\end{equation}

\noindent for all $N$.

\end{lemma}

Now putting together Lemma \ref{le:bigk} and Lemma \ref{le:smallk}, we see that for some constant $C>0$, $\E|s_k(X_N)|\leq C(k+1)^{2}$. As in \cite{fks,hko}, this implies tightness in $\mathcal{S}_{-\alpha}$ for $\alpha>3$. For completeness, we'll give a brief sketch of a proof of this fact.

\begin{proposition}
For $\alpha>3$, $(X_N)_N$ is tight in $\mathcal{S}_{-\alpha}$.
\end{proposition}

\begin{proof}
From the definition of $\mathcal{S}_{\alpha}$, $\mathcal{S}_{-\alpha'}\subset \mathcal{S}_{-\alpha}$ for $\alpha'<\alpha$ and one can check (see e.g. \cite[Theorem 8.3]{kress}) that the embedding is a compact operator. Thus if we consider any ball in $\mathcal{S}_{-\alpha'}$, say $B_{R,\alpha'}=\lbrace ||\psi||_{-\alpha'}\leq R\rbrace$, this is a compact subset of $\mathcal{S}_{-\alpha}$ for $\alpha'<\alpha$. Now when we consider $(X_N)$ as elements in $\mathcal{S}_{-\alpha}$ and $B_{R,\alpha'}$ as a compact subset of $\mathcal{S}_{-\alpha}$, we have by Markov's inequality

\begin{equation}
\mathbb{P}(X_N\notin B_{R,\alpha'})=\mathbb{P}(||X_N||_{-\alpha'}\geq R)\leq \frac{\E||X_N||_{-\alpha'}}{R}.
\end{equation}

Now let $3<\alpha'<\alpha$. We have (recall \eqref{eq:tight})

\begin{align}
\notag\E||X_N||_{-\alpha'}&\leq \sum_{k=0}^{\infty}\E|s_k(X_N)|(1+k^{2})^{-\alpha'/2}\\
&\leq C\sum_{k=1}^{\infty} k^{2} k^{-\alpha'}\\
\notag &< \infty
\end{align}

\noindent so we see that $(X_N)$ is tight in $\mathcal{S}_{-\alpha}$.

\end{proof}

\subsection{Proof of Theorem \ref{th:main}}

While we've already mentioned that tightness and convergence on cylinder sets implies weak convergence, let us still sketch some of the details here. As mentioned, Prohorov's theorem is of key importance here. By it, all subsequences of $(X_N)_N$ have a weakly converging subsequence in $\mathcal{S}_{-\alpha}$ for $\alpha>3$. One then needs to check that all of these limits have the same law, or equivalently, that their finite dimensional distributions agree. Due to the linear structure of the Sobolev space, this means that we need to check that if $X^{(1)}$ and $X^{(2)}$ are two limit points of $(X_N)_N$, then for each $f\in \mathcal{S}_\alpha$

\begin{equation}
X^{(1)}(f)\stackrel{d}{=}X^{(2)}(f).
\end{equation}

For $i=1,2$, Let $N_k^{(i)}$ be sequences such that $X_{N_k^{(i)}}\to X^{(i)}$ weakly as $k\to \infty$. We then fix some $M\in \Z_+$ and write 

\begin{align}
\notag X_{N_k^{(i)}}(f)&=\sum_{j=0}^{\infty}s_j\left(X_{N_k^{(i)}}\right)s_j(f)\\
&=\sum_{j=0}^{M}s_j\left(X_{N_k^{(i)}}\right)s_j(f)+ \sum_{j=M+1}^{\infty}s_j\left(X_{N_k^{(i)}}\right)s_j(f).
\end{align}

Now if we first let $k\to\infty$ and then $M\to\infty$, the first term converges weakly to $X(f)$ (Lemma \ref{le:fdd}). For the second term we have 

\begin{align}
\notag \E\left|\sum_{j=M+1}^{\infty}s_j\left(X_{N_k^{(i)}}\right)s_j(f)\right|&\leq ||f||_{\alpha}\E\sqrt{\sum_{j=M+1}^{\infty}\left|s_j\left(X_{N_k^{(i)}}\right)\right|^{2}(1+j^{2})^{-\alpha}}\\
&\leq C||f||_{\alpha}\sum_{j=M+1}^{\infty}\E\left|s_j\left(X_{N_k^{(i)}}\right)\right| j^{-\alpha}
\end{align}

By the bound $\E|s_j(X_N)|\leq C j^{2}$, we see that as $k\to\infty$ and then $M\to\infty$

\begin{equation}
\sum_{j=M+1}^{\infty}s_j\left(X_{N_k^{(i)}}\right)s_j(f)\stackrel{d}{\to} 0.
\end{equation}

By Slutsky's theorem we conclude that

\begin{equation}
X_{N_k^{(i)}}(f)\stackrel{d}{\to} X(f)
\end{equation}

\noindent so that $X^{(1)}\stackrel{d}{=}X^{(2)}$.

\end{document}